\documentclass[review,12pt]{elsarticle}
\UseRawInputEncoding

\biboptions{numbers,sort&compress}
\usepackage{amssymb}
\usepackage{graphicx}
\usepackage{amsmath}
\usepackage{lineno,hyperref}
\usepackage{amsmath}
\usepackage{appendix}
\usepackage{graphicx}
\usepackage{setspace}
\usepackage{geometry}
\numberwithin{equation}{section}
\newtheorem{theorem}{Theorem}[section]

\newtheorem{proposition}{Proposition}[section]
\newtheorem{remark}{Remark}[section]

\begin{document}

\begin{frontmatter}



\title{Dynamics for a Ratio-dependent Prey-predator Model with Different Free Boundaries \tnoteref{mytitlenote}}
\tnotetext[mytitlenote]{This work was supported by \href{http://www.ctan.org/tex-archive/macros/latex/contrib/elsarticle}{NSFC 12071316}.}
\author[]{Lingyu Liu}
\ead{liu\_lingyu@foxmail.com}

\address{Department of Mathematics, Sichuan University, Chengdu 610065, PR China}

\begin{abstract}
In this paper, we study the dynamics of the ratio-dependent type prey-predator model with different free boundaries. The two free boundaries, determined by prey and predator respectively, implying that they may intersect each other as time evolves, are used to describe the spreading of prey and predator. We mainly investigate the longtime behaviors of predator and prey. Then some sufficient conditions for spreading and vanishing are given. In addition, when spreading occurs, some estimates of asymptotic spreading speeds of $u$, $v$ and  asymptotic speeds of $h$, $g$ are provided.
\end{abstract}

\begin{keyword}
\texttt{free boundary\sep ratio-dependent model\sep spread\sep vanish\sep criteria \sep asymptotic speed}
\end{keyword}

\end{frontmatter}


\section{Introduction}
The dynamic of species is a fundamental content in ecology. In order to study this problem, numerous researches have been done from a mathematical viewpoint. In this paper, we consider a ratio-dependent prey-predator model with different double free boundaries as follows,
\begin{equation}\label{Q}
\left\{
\begin{array}{ll}
u_t-u_{xx}=u(\lambda -u-\frac{bv}{u+mv}),    &t>0 ,~0<x<h(t),\\
v_t -dv_{xx}=v(1-v+\frac{cu}{u+mv}),       &t>0 ,~0<x<g(t),\\
u_x=v_x=0,                          &t\geq0,~x=0,\\
u(x)=0~ for~x\geq h(t),~v(x)=0~for~x\geq g(t),                             &t\geq0,\\
h^{\prime}(t)=-\mu u_x(t,h(t)),~g^{\prime}(t)=-\rho v_x(t,g(t)), &t\geq0,\\
u(0,x)=u_0(x)~in~[0,h_0] ,~ v(0,x)=v_0(x)~in~[0,g_0],\\
h(0)=h_0\geq g(0)=g_0>0,
\end{array}
\right.
\end{equation}
where $\lambda$, $b$, $m$, $d$, $c$, $\mu$ ($\rho$), $h_0$ ($g_0$) are positive constants that stand for prey intrinsic growth, capturing rate, half capturing saturation constant, dispersal rate, conversion rate, moving parameter of prey (predator), initial habitat of prey (predator), respectively. $u$ and $v$ stand for prey and predator density, respectively. Term $\frac{u}{u+mv}$ is the functional response.
The initial functions $u_0(x)$ and $v_0(x)$ satisfy
$$
u_0 \in {C}^2([0,h_0]),~u_0(x)>0,~x\in[0,h_0),
$$
$$
v_0\in C^2([0,g_0]),~v_0(x)>0,~x\in[0,g_0).
$$
The formulas $x=h(t)$ and $x=g(t)$ are the free boundaries to be solved. The free boundary conditions in (\ref{Q}) are in accord with the one-phase Stefan condition which stems from the investigations of melting of ice \cite{L1971The} and wound healing \cite{15}. In \cite{Du1}, Du and Lou made a first attempt to introduce the Stefan condition into the study of spreading populations. Afterwards, Du and his cooperators (\cite{Du1,Du2,Du3,Du4,Du5,Du6,Du2001LOGISTIC,2012Du}), Wang and his cooperators (\cite{wang5,wang6,wang7,wang8,wang9,wang10,wang11,wang12,Wangzhao2014a}) and other scholars did lots of research on the spreading and vanishing of species.

The two free boundaries $h(t)$ and $g(t)$ may intersect each other due to their separate independence, which may be difficult to understand the whole dynamics of model (\ref{Q}).
In \cite{wang12}, Wang and Zhang studied a similar spreading mechanism but for a diffusive Lotka-Volterra type prey-predator model. The response function of the model mentioned above depends only on prey, which illustrates that the predation's behavior is determined solely by prey. Nevertheless, there is mounting evidence that, in some specific ecological environments, especially when a predator has to actively seek, share and plunder prey, a more rational prey-predator model should be ratio-dependent, i.e., the response function is ratio-dependent. Thus, we attempt to study the ratio-dependent prey-predator model with different double free boundaries in this paper.

The main objective of this paper is to investigate the dynamics of problem (\ref{Q}) and attempt to explain its biological significance. Same to \cite[Theorem 2.1, Lemma 2.1, Theorem 2.2]{wang12} , for the global existence, uniqueness and regularity of solution $(u,v,h,g)$, we have the following theorem.

\begin{theorem}\label{th3}
The problem (\ref{Q}) has a unique global solution. Besides, for any $\alpha\in(0,1)$,
$$
(u,v,h,g)\in{C}^{\frac{1+\alpha}{2},1+\alpha}(\overline D_\infty^{h(t)})\times{C}^{\frac{1+\alpha}{2},1+\alpha}(\overline D_\infty^{g(t)})\times [{C}^{1+\frac{\alpha}{2}}([0,T])]^2,
$$
where $D_\infty^s=\{(t,x)\in\mathbb{R}^2:t>0,0<x<s(t)\}$. Moreover,
$$
0<u(t,x)\leq \max\{\lambda,||u_0||_{\infty}\}=:M_1,~t>0,~0\leq x\leq h(t),,\
$$
$$
0<v(t,x)\leq \max\{1+c,||v_0||_{\infty}\}=:M_2,~t>0,~0\leq x\leq g(t),
$$
$$
0\leq h^{\prime}(t)\leq 2\mu\max\{M_1\sqrt{\lambda/2},-\min\limits_{[0,h_0]}u_0^{\prime}(x)\}=:M_3,~t>0,
$$
$$
0\leq g^{\prime}(t)\leq 2\rho\max\{M_2\sqrt{(1+c)/2d},-\min\limits_{[0,g_0]}v_0^{\prime}(x)\}=:M_4,~t>0.
$$
\end{theorem}

This paper is organized as follows. In Section 2, we give some fundamental results which will be used in the following sections. Section 3 is concerned with the longtime behaviors of $u$ and $v$ under different conditions. In Section 4, we provide the conditions for spreading and vanishing of prey and predator. In Section 5, some estimates of the asymptotic spreading speeds of $u$, $v$ and asymptotic speeds of $g$, $h$ are provided. According to the results above, Section 6 makes a summary and gives some practical significance.

\section{Preliminaries}
We first consider a diffusive logistic problem with a free boundary
\begin{equation}\label{27}
\begin{cases}
\omega_t-d\omega_{xx}=\omega(\theta-\omega),~~~~&t>0,~0<x<s(t),\\
\omega_x(t,0)=0,~\omega(t,s(t))=0,           &t\geq0,\\
s^{\prime}(t)=-\beta\omega_x(t,s(t)),      &t\geq0,\\
\omega(0,x)=\omega_0(x),~s(0)=s_0,           &0\leq x\leq s_0,
\end{cases}
\end{equation}
where $d$, $\theta$, $\beta$, $s_0$ are positive constants.
\begin{proposition}\label{p2}{\upshape(\cite{Du1})}
The problem (\ref{27}) has a unique global solution $(\omega,s)$ and $\lim\limits _{t\rightarrow\infty}s(t)=s_{\infty}$ exists. Besides,

{\upshape(i)} if $s_0\geq\frac{\pi}{2}\sqrt{\frac{d}{\theta}}$, then $s_{\infty}=\infty$ for all $\beta>0$;

{\upshape(ii)} if $s_0<\frac{\pi}{2}\sqrt{\frac{d}{\theta}}$, then there is a positive constant $\beta(d,\theta,s_0,\omega_0)$ such that $s_{\infty}=\infty$ if $\beta>\beta(d,\theta,s_0,\omega_0)$, and $s_{\infty}<\infty$ if $\beta<\beta(d,\theta,s_0,\omega_0)$;

{\upshape(iii)} if $s_{\infty}=\infty$, then $\lim\limits_{t\rightarrow\infty}\omega(t,x)=\theta$ uniformly in any compact subset of $[0,\infty)$.
\end{proposition}

In addition, it has proved that the expanding front $s(t)$ moves at a constant speed for large time, i.e., $s(t)=(c+o(1))t$ as $t\rightarrow\infty$ in \cite{Du1}. And the spreading speed $c$ is determined by the following auxiliary elliptic problem
\begin{equation}\label{3}
\left\{
\begin{array}{ll}
dq^{\prime\prime}-cq^{\prime}+q(\theta-q)=0,~~0<y<\infty,\\
q(0)=0,~~q^{\prime}(0)=c/\beta,~~q(\infty)=\theta,\\
c\in(0,2\sqrt{\theta d});~~q^{\prime}(y)>0,~~0<y<\infty,
\end{array}
\right.
\end{equation}
where $d$, $c$, $\theta$, $\beta$ are given positive constants.

\begin{proposition}\label{p1}\upshape{(\citep{Du6})}
The problem (\ref{3}) has a unique solution $(q(y),c)$ and $c(\beta,d,\theta)$ is strictly increasing in $\beta$ and $\theta$, respectively. Moreover,
\begin{equation}\label{4}
\lim\limits_{\frac{\theta\beta}{d}\rightarrow\infty}\frac{c(\beta,d,\theta)}{\sqrt{\theta d}}=2,~~\lim\limits_{\frac{\theta\beta}{d}\rightarrow 0}\frac{c(\beta,d,\theta)}{\sqrt{\theta d}}\frac{d}{\theta\beta}=\frac{1}{\sqrt{3}}.
\end{equation}
\end{proposition}

\begin{proposition}{\upshape(\cite{wang5})}\label{p3}
Let $(q,c)$ be the unique solution of (\ref{3}). If $(\omega,s)$ is a solution of (\ref{27}) for which spreading occurs, then there exists $H\in\mathbb{R}$ such that
$$
\lim\limits_{t\rightarrow\infty}(s(t)-ct-H)=0,~\lim\limits_{t\rightarrow\infty}s^{\prime}(t)=c,
$$
$$
\lim\limits_{t\rightarrow\infty}||\omega(t,x)-q(ct+H-x)||_{L^{\infty}([0,s(t)])}=0.
$$

\end{proposition}

\section{Longtime behaviors of $u$ and $v$}
Since the two free boundaries $h(t)$ and $g(t)$ are monotonically increasing and may intersect each other, there exist four cases: (i) $h_{\infty}=g_{\infty}=\infty$; (ii) $h_{\infty}=\infty$ and $g_{\infty}<\infty$; (iii) $h_{\infty}<\infty$ and $g_{\infty}=\infty$; (iv) $h_{\infty}<\infty$ and $g_{\infty}<\infty$. In this section, we will study the limits of $u(t,x)$ and $v(t,x)$ as $t\rightarrow\infty$ under the above four cases.

\begin{theorem}\label{th6}
Assume that $h_{\infty}<\infty$ $(g_{\infty}<\infty)$. Then
$$
\lim\limits_{t\rightarrow\infty}\max\limits_{0\leq x\leq h(t)}u(t,x)=0,~(\lim\limits_{t\rightarrow\infty}\max\limits_{0\leq x\leq g(t)}v(t,x)=0).
$$
\end{theorem}
\begin{proof}
The conclusion can be deduced by Theorem \ref{th3} and \cite[Lemma 4.1]{wang12} directly.
\end{proof}

Similar to \cite[Theorem 3.2]{liu2020free}, we can prove Theorem \ref{th7} as follows. So we omit the proof.
\begin{theorem}\label{th7}
Assume that $h_{\infty}=g_{\infty}=\infty$. Then $\lim\limits_{t\rightarrow\infty}(u(t,x),v(t,x))$ is determined by
\begin{equation}\label{90}
\begin{cases}
\lambda-u-\frac{bv}{u+mv}=0,\\
1-v+\frac{cu}{u+mv}=0.
\end{cases}
\end{equation}
Further calculations give, when $0<m\lambda-b<b/c$,
$$\lim\limits_{t\rightarrow\infty}u(t,x)=u^*:=\frac{A+\sqrt{\Delta_1}}{2(b +c m^2)},~\lim\limits_{t\rightarrow\infty}v(t,x)=v^*:=\frac{u^*(\lambda-u^*)}{b-m(\lambda-u^*)},$$
where $A=\lambda(2cm^2+b)-mb(1+2c)$,
$\Delta_1=A^2+4(b+cm^2)[b(1+c)-mc\lambda](m\lambda-b)$.
\end{theorem}

\begin{theorem}\label{th8}
{\upshape(i)} If $h_{\infty}=\infty$ and $g_{\infty}<\infty$, then
$$
\lim\limits_{t\rightarrow\infty} u(t,x)=\lambda ~~uniformly~ in ~any ~compact~ subset~ of~ [0,\infty).
$$

{\upshape(ii)} If $h_{\infty}<\infty$ and $g_{\infty}=\infty$, then
$$
\lim\limits_{t\rightarrow\infty} v(t,x)=1~~uniformly~ in~ any~ compact~ subset~ of~ [0,\infty).
$$

\end{theorem}

\begin{proof}
We only prove case (i) as proof of (ii) can be verified in the same way. According to the conclusions of the logistic equation and comparison principle, we have
\begin{equation}\label{29}
\lambda-b/m\leq\liminf_{t\rightarrow\infty} u(t,x),~\limsup_{t\rightarrow\infty} u(t,x)\leq\lambda
\end{equation}
uniformly in any compact subset of $[0,\infty)$ when $h_\infty=\infty$.
Note that $\lim\limits_{t\rightarrow\infty}\max\limits_{0\leq x\leq g(t)}v(t,x)=0$ for $g_{\infty}<\infty$ by Theorem \ref{th6} and $v(t,x)=0$ for $x>g(t)$. Owing to the former inequality of (\ref{29}), for any given $0<\varepsilon\ll1$, there exists $T\gg1$ such that $v(t,x)<\varepsilon$ for all $t\geq T$ and $x\geq 0$. Then $u$ satisfies
$$
\begin{cases}
u_t-du_{xx}\geq(\lambda-\frac{\varepsilon b}{\lambda-b/m+\varepsilon m})u-u^2,~~&t\geq T,~0<x<h(t),\\
u_x(t,0)=0,~u(t,h(t))=0,                                              &t\geq T,
\end{cases}
$$
which implies that $\liminf\limits_{t\rightarrow\infty}u(t,x)\geq \lambda-\frac{\varepsilon b}{\lambda-b/m+\varepsilon m}$ uniformly in any compact subset of $[0,\infty)$.
By the arbitrariness of $\varepsilon$, we have
$$
\liminf_{t\rightarrow\infty} u(t,x)\geq \lambda~~uniformly~ in~ any~ compact~ subset~ of~ [0,\infty).
$$
Combining with the latter inequality of (\ref{29}), we can get case (i).
\end{proof}

\section{Conditions for spreading and vanishing}
In any case, the following inequalities are always true
\begin{equation}\label{28}
\begin{cases}
\max\{0,\lambda-b/m\}u-u^2\leq\lambda u-u^2-\frac{buv}{u+mv}\leq\lambda u-u^2,\\
v-v^2\leq v-v^2+\frac{cuv}{u+mv}\leq(1+c)v-v^2.
\end{cases}
\end{equation}
Denote
$$
\mu^*:=\beta(1,\lambda,h_0,u_0),~\mu_*:=\beta(1,\lambda-b/m,h_0,u_0),
$$
and
$$
\rho^*:=\beta(d,1+c,g_0,v_0),~\rho_*:=\beta(d,1,g_0,v_0).
$$
By use of Proposition \ref{p2}, we have the following theorem immediately.
\begin{theorem}\label{th1}
Assume that $m\lambda>b$.

{\upshape(i)} If $h_0<\frac{\pi}{2}\sqrt{\frac{1}{\lambda}}$ and $\mu\leq\mu_*$, then $h_{\infty}<\infty$;

{\upshape(ii)}If $h_0\geq\frac{\pi}{2}\sqrt{\frac{m}{m\lambda-b}}$ or $h_0<\frac{\pi}{2}\sqrt{\frac{m}{m\lambda-b}}$ and $\mu>\mu^*$, then $h_{\infty}=\infty$;

{\upshape(iii)}If $g_0<\frac{\pi}{2}\sqrt{\frac{d}{1+c}}$ and $\rho\leq\rho_*$, then $g_{\infty}<\infty$;

{\upshape(iv)}If $g_0\geq\frac{\pi}{2}\sqrt{d}$ or $g_0<\frac{\pi}{2}\sqrt{d}$ and $\rho>\rho^*$, then $g_{\infty}=\infty$.
\end{theorem}

Next, we will study the relationship between $h(t)$ and $g(t)$ when $t\rightarrow\infty$.
In view of Theorem \ref{th3}, for a constant $s\geq M_4$, we have $0<g(t)\leq st+g_0$.
\begin{theorem}\label{th2}
Let $\lambda$, $\mu$ be fixed. Then there exists $0<\overline s<\sqrt{2\lambda}$ such that when
$$
0<s<\overline s,~h_0-g_0>\frac{2\pi}{\sqrt{2\lambda-s^2}}=:L_s,
$$
we have $h(t)\geq st+g_0+L_s$, which implies
$h(t)>g(t)$ for all $t\geq0$ and $h_{\infty}\rightarrow\infty$ as $t\rightarrow\infty$.
\end{theorem}

\begin{proof}
Inspired by \cite[Theorem 5.3]{wang12}, for $h(t)>st+g_0$, we define
$$
y=x-st-g_0,~\varphi(t,y)=u(t,x),~\eta(t)=h(t)-st-g_0.
$$
Then $\varphi(t,y)>0$ for $t\geq0$ and $0\leq y<\eta(t)$. Evidently, $v\geq0$ implies that $x\geq g(t)$. Note that $v(t,x)=0$ for $x\geq g(t)$, so $\varphi$ satisfies
$$
\begin{cases}
\varphi_t-\varphi_{xx}-s\varphi_y=\varphi(\lambda-\varphi),~~~~~~&t>0,~0<y<\eta(t),\\
\varphi(t,0)=u(t,st+g_0),~\varphi(t,\eta(t))=0,                  &t\geq0,\\
\varphi(0,y)=u_0(y+g_0),                                         &0\leq y\leq h_0-g_0.
\end{cases}
$$
Let $\sigma$ be the principle eigenvalue of
\begin{equation}\label{1}
\begin{cases}
-\phi^{\prime\prime}-s\phi^{\prime}-\lambda\phi=\sigma\phi,~~&0<x<L,\\
\phi(0)=\phi(L)=0.
\end{cases}
\end{equation}
Then the relation between $\sigma$ and $L$ satisfies
$$
\frac{\pi}{L}=\frac{\sqrt{4(\lambda+\sigma)-s^2}}{2}.
$$
For $0<s<\sqrt{2\lambda}$, choose $\sigma=-\lambda/2$ and define
$$
L_s=\frac{2\pi}{\sqrt{2\lambda-s^2}},~\phi(y)=e^{-\frac{s}{2}y} \sin \frac{\pi}{L_{s}}y,
$$
and then $(L_s,\phi)$ satisfies (\ref{1}) with $\sigma=-\frac{\lambda}{2}$ and $L=L_s$. Assume $h_0-g_0>L_s$ and define
$$
\delta_s=\min\left\{\inf\limits_{(0,L_s)}\frac{\varphi(0,y)}{\phi(y)},\frac{\lambda}{2}\inf\limits_{(0,L_s)}\frac{1}{\phi(y)}\right\},~\psi(y)=\delta_s\phi(y).
$$
Then $0<\delta_s<\infty$ and $\psi(y)\leq\varphi(0,y)$ in $[0,L_s]$. Moreover, $\psi(y)$ satisfies
$$
\begin{cases}
-\psi^{\prime\prime}-s\psi^{\prime}\leq\psi(\lambda-\psi),~~~~&0<x<L_s,\\
\psi(0)=\psi(L_s)=0.
\end{cases}
$$
Take a maximal $\overline s\in(0,\sqrt{2\lambda})$ so that
\begin{equation}\label{2}
s<\mu\delta_s\frac{\pi}{L_s}e^{-\frac{s}{2}L_s},~~\forall s\in(0,\overline s).
\end{equation}
For any given $s\in(0,\overline s)$, we claim that $\eta(t)>L_s$ for $t\geq0$, which leads to
$$
h(t)\geq st+g_0+L_s\rightarrow\infty, ~~t\rightarrow\infty.
$$
Note that $\eta(0)=h_0-g_0>L_s$, if our claim is invalid, then we can find a $t_0>0$ such that $\eta(t)>L_s$ for all $0<t<t_0$ and $\eta(t_0)=L_s$. So $\eta^{\prime}(t_0)\leq0$ and $h^{\prime}(t_0)\leq s$. Besides, by the comparison principle, we have $\varphi(t,y)\geq\psi(y)$ in $[0,t_0]\times[0,L_s]$. So we have $\varphi(t_0,y)\geq\psi(y)$ in $[0,L_s]$. Thanks to $\varphi(t_0,L_s)=\varphi(t_0,\eta(t_0))=0=\psi(L_s)$, we get
$$
\varphi_y(t_0,L_s)\leq\psi^{\prime}(L_s)=-\delta_s\frac{\pi}{L_s}e^{-\frac{s}{2}L_s}.
$$
A direct calculation deduces that
$$
\begin{aligned}
 s&\geq h^{\prime}(t_0)=-\mu u_x(t_0,h(t_0))=-\mu\varphi_y(t_0,\eta(t_0))=-\mu\varphi_y(t_0,L_s)\\
  &\geq\mu\delta_s\frac{\pi}{L_s}e^{-\frac{s}{2}L_s},
\end{aligned}
$$
which contradicts (\ref{2}). The proof is complete.
\end{proof}

Combining Theorem \ref{th1}(iv) and Theorem \ref{th2}, we get the following remark immediately.
\begin{remark}\label{rm1}
Assume that $\lambda$, $\mu$, $s$, $h_0$, $g_0$ satisfy the conditions of Theorem \ref{th2}. If $g_0\geq\frac{\pi}{2}\sqrt{d}$, then $g_{\infty}=\infty$, which implies $h_{\infty}=\infty$ for all $\mu>0$.
\end{remark}

According to Proposition \ref{p1}, we have
$$
\lim\limits_{\mu\rightarrow0}c(\mu,1,\lambda)=0,~\lim\limits_{\rho\rightarrow\infty}c(\rho,d,1)=2\sqrt{d}.
$$
Note that $c(\beta,d,\theta)$ is strictly increasing in $\beta$, so there exist positive constants $\overline\mu$ and $\overline\rho$ such that $c(\mu,1,\lambda)<c(\rho,d,1)$ for all $0<\mu\leq\overline\mu$ and $\rho\geq\overline \rho$. Define
$$
\mathcal{F}:=\{(\mu,\rho):\mu,\rho>0,~c(\mu,1,\lambda)<c(\rho,d,1)\},
$$
then $(0,\overline\mu]\times[\overline\rho,\infty)\subset\mathcal{F}$.

\begin{theorem}\label{th4}
Assume that $(\mu,\rho)\in\mathcal{F}$. If $\lambda^2+m\lambda<b$ and $g_{\infty}=\infty$, then $h_{\infty}<\infty$.
\end{theorem}

\begin{proof}
As a result of $\lambda^2+m\lambda<b$, there exists $0<\varepsilon\ll 1$ such that \begin{equation}\label{6}
\lambda^2+(m\lambda-b)(1-\varepsilon)\leq0.
\end{equation}
Owing to $g_{\infty}=\infty$, there exists $t_1>0$ such that $g(t_1)>\frac{\pi}{2}\sqrt{d}$. Let $(\underline v,\underline g)$ be the unique solution of
$$
\begin{cases}
\underline v_t-d\underline v_{xx}=\underline v-\underline v^2,~~~~&t>t_1,~0<x<\underline g(t),\\
\underline v_x(t,0)=0,~\underline v(t,\underline g(t))=0,                &t>t_1,\\
\underline g^{\prime}=-\rho\underline v_x(t,\underline g(t)),         &t\geq t_1,\\
\underline v(t_1,x)=v_0(t_1,x),~\underline g(t_1)=g(t_1),                    &0\leq x\leq g(t_1).
\end{cases}
$$
By the comparison principle, it yields that $v(t,x)\geq\underline v(t,x)$, $g(t)\geq \underline g(t)$ for $t\geq t_1$, $0\leq x\leq\underline g(t)$. It follows from Proposition \ref{p2} (i) that
$\underline g(\infty)=\infty$. In addition, by Proposition \ref{p3}, there exists $\mathcal{S}_1\in\mathbb{R}$ such that
\begin{equation}\label{5}
\underline g(t)-\tilde{c} t\rightarrow \mathcal{S}_1,~||\underline v(t,x)-q_1(\tilde{c}t+\mathcal{S}_1-x)||_{L^{\infty}([0,\underline g(t)])}\rightarrow 0,~as~t\rightarrow\infty,
\end{equation}
where $(q_1(y),\tilde{c})$ is the unique solution of (\ref{3}) with $(\beta,d,\theta)=(\rho,d,1)$, i.e., $\tilde{c}=c(\rho,d,1)$.

In view of $(\mu,\rho)\in\mathcal{F}$, we have $\tilde{c}>c(\mu,1,\lambda)$ which implies that $\underline g(t)-h(t)\geq \underline g(t)-\overline h(t)\rightarrow\infty$. So, we get $\min\limits_{0\leq x\leq h(t)}q_1(\tilde{c}t+\mathcal{S}_1-x)\rightarrow1$ as $t\rightarrow\infty$. Thus, by (\ref{5}), we have $\lim\limits_{t\rightarrow\infty}\min\limits_{0\leq x\leq h(t)}\underline v(t,x)=1$. There exists $t_2>t_1$ such that
\begin{equation}\label{30}
\underline v(t,x)>1-\varepsilon,~~t\geq t_2,~0\leq x\leq h(t).
\end{equation}
Besides, it follows from Theorem \ref{th3} that $u(t,x)\leq \lambda$ for $t>t_2$, $0\leq x\leq h(t)$.
By the comparison principle and (\ref{6}), we then have
\begin{equation}\label{7}
\begin{aligned}
  \lambda-u-\frac{bv}{u+mv}
&<\lambda-\frac{b(1-\varepsilon)}{\lambda+m(1-\varepsilon)}-u\\
&=\frac{\lambda^2+(m\lambda-b)(1-\varepsilon)}{\lambda+m(1-\varepsilon)}-u\\
&<0
\end{aligned}
\end{equation}
for all $t\geq t_2$ and $0\leq x\leq h(t)$. Direct calculation gives
$$
\begin{aligned}
\frac{d}{dt}\int_0^{h(t)}u(t,x)dx&=\int_0^{h(t)}u_t(t,x)dx+h^{\prime}(t)u(t,h(t))\\
                                 &=\int_0^{h(t)}u_{xx}(t,x)dx+\int_0^{h(t)}(\lambda u-u^2-\frac{buv}{u+mv})dx\\
                                 &=-\frac{1}{\mu}h^{\prime}(t)+\int_0^{h(t)}(\lambda u-u^2-\frac{buv}{u+mv} )dx.
\end{aligned}
$$
By integrating from $t_2$ to $t$ firstly and applying (\ref{7}) secondly, we have
$$
\begin{aligned}
\int_0^{h(t)}u(t,x)dx
&=\int_0^{h(t_2)}u(t_2,x)+\frac{1}{\mu}(h(t_2)-h(t))+\int_{t_2}^t\int_0^{h(\tau)}(\lambda u-u^2-\frac{buv}{u+mv})dxd\tau\\
&\leq \int_0^{h(t_2)}u(t_2,x)+\frac{1}{\mu}(h(t_2)-h(t)),
\end{aligned}
$$
which implies
$$
h(t)\leq h(t_2)+\mu\int_0^{h(t_2)}u(t_2,x)dx.
$$
Hence, $h_{\infty}<\infty$. This completes the proof.
\end{proof}

\section{Asymptotic spreading speeds of $u$, $v$ and asymptotic speeds of $g$, $h$}

This section is devoted to dealing with the asymptotic spreading speeds of $u$, $v$ and asymptotic speeds of $g$, $h$ when spreading occurs ($h_{\infty}=g_{\infty}=\infty$). In this section, we always assume that $m\lambda>b$.

According to (\ref{28}), Proposition \ref{p1} and the comparison principle, we have
\begin{equation}\label{12}
\limsup_{t\rightarrow\infty}\frac{h(t)}{t}\leq c(\mu,1,\lambda)=:\overline c_\mu,~\liminf_{t\rightarrow\infty}\frac{h(t)}{t}\geq c(\mu,1,\lambda-b/m)=:\underline c_{\mu},
\end{equation}
\begin{equation}\label{13}
\limsup_{t\rightarrow\infty}\frac{g(t)}{t}\leq c(\rho,d,1+c)=:\overline c_\rho,~\liminf_{t\rightarrow\infty}\frac{g(t)}{t}\geq c(\rho,d,1)=:\underline c_\rho.
\end{equation}
Denote
$$
c_1=2\sqrt{\lambda-b/m},~c_2=2\sqrt{\lambda},~c_3=2\sqrt{d},~c_4=2\sqrt{d(1+c)}.
$$

\begin{theorem}\label{th5}
For any given $0<\varepsilon\ll1$, there exist $\mu_\varepsilon$, $\rho_\varepsilon$, $T\gg1$ such that when $\rho\geq\rho_\varepsilon$ and $\mu\geq\mu_\varepsilon$,
\begin{equation}\label{8}
u(t,x)=0~for~t\geq T,~x\geq(c_2+\varepsilon)t,
\end{equation}
\begin{equation}\label{9}
v(t,x)=0~for~t\geq T,~x\geq(c_4+\varepsilon)t,
\end{equation}
\begin{equation}\label{10}
\liminf_{t\rightarrow\infty}\min_{0\leq x\leq(c_1-\varepsilon)t}u(t,x)\geq\lambda-b/m,
\end{equation}
\begin{equation}\label{11}
\liminf_{t\rightarrow\infty}\min_{0\leq x\leq (c_3-\varepsilon)t}v(t,x)\geq1.
\end{equation}
\end{theorem}

\begin{proof}
It follows from Proposition \ref{p1} that
$$
\lim\limits_{\mu\rightarrow\infty}\underline c_\mu=c_1,~\lim\limits_{\mu\rightarrow\infty}\overline c_\mu=c_2,~\lim\limits_{\rho\rightarrow\infty}\underline c_\rho=c_3,~\lim\limits_{\rho\rightarrow\infty}\overline c_\rho=c_4.
$$
Take notice of (\ref{12}) and (\ref{13}), so for any given $0<\varepsilon\ll1$, there exist $\mu_\varepsilon$, $\rho_\varepsilon\gg1$ such that
\begin{equation}\label{88}
c_1-\varepsilon/2<\underline c_\mu\leq\liminf_{t\rightarrow\infty}\frac{h(t)}{t},~\limsup_{t\rightarrow\infty}\frac{h(t)}{t}\leq\overline c_\mu<c_2+\varepsilon/4,
\end{equation}
\begin{equation}\label{89}
c_3-\varepsilon/2<\underline c_\rho\leq\liminf_{t\rightarrow\infty}\frac{g(t)}{t},~\limsup_{t\rightarrow\infty}\frac{g(t)}{t}\leq\overline c_\rho<c_4+\varepsilon/2.
\end{equation}
Thus, there exists a $\tau_1\gg1$ such that , for all $t\geq\tau_1$, $\mu\geq\mu_\varepsilon$ and $\rho\geq\rho_\varepsilon$,
\begin{equation}\label{17}
(c_1-\varepsilon)t<h(t)<(c_2+\varepsilon/2)t,~(c_3-\varepsilon)t<g(t)<(c_4+\varepsilon)t.
\end{equation}
It is obvious that  (\ref{8}) and (\ref{9}) hold.

The proof of (\ref{10}). Let $(\underline u,\underline h)$ be the unique solution of (\ref{27}) with $({\rho},d,a,b)=({\rho},1,{\lambda}-b/m,1)$.
Then $h(t)\geq\underline h(t)$, $u(t,x)\geq\underline u(t,x)$ for $t>0$ and $0<x\leq\underline h(t)$ by Proposition \ref{rm1}. Make use of \cite[Theorem 3.1]{wang5}, we get
$$
\lim_{t\rightarrow\infty}(\underline h(t)-c^*t)=H\in\mathbb{R},~\lim_{t\rightarrow\infty}||\underline u(t,x)-q^*(c^*t+H-x)||_{L^{\infty}([0,\underline h(t)])}=0,
$$
where $(q^*(y),c^*)$ is the unique solution of (\ref{3}) with $(\beta,d,\theta)=({\rho},1,{\lambda}-{b}/{m})$, i.e., $c^*=c({\rho},1,{\lambda}-b/m)>c_1-\frac{\varepsilon}{2}$. Obviously, $\min_{[0,(c_1-\frac{\varepsilon}{2})t]}(c^*t+H-x)\rightarrow\infty$ as $t\rightarrow\infty$.
Owing to $q^*(y)\nearrow{\lambda}-{b}/{m}$ as $y\nearrow\infty$, we have $\min_{[0,(c_1-\frac{\varepsilon}{2})t]}q^*(c^*t+H-x)\rightarrow{\lambda}-b/m$ as $t\rightarrow\infty$. So $\min_{[0,(c_1-\frac{\varepsilon}{2})t]}\underline u(t,x)\rightarrow{\lambda}-b/m$ as $t\rightarrow\infty$. Thus (\ref{10}) holds due to $u(t,x)\geq\underline u(t,x)$ for $t>0$ and $0<x\leq\underline h(t)$.

Similarity, we can prove that (\ref{11}) holds.
\end{proof}

In the following we give a further study on $u$, $h$ and $v$, $g$.
\begin{theorem}\label{th9}
If $d(1+c)<\lambda-b/m$, then for any given $\varepsilon>0$, there exist $\rho_{\varepsilon}$, $\mu_{\varepsilon}$, $T\gg1$ such that, when $\rho\geq\rho_{\varepsilon}$ and $\mu\geq\mu_{\varepsilon}$, we have
\begin{equation}\label{26}
\lim\limits_{\mu\rightarrow\infty}\lim\limits_{t\rightarrow\infty}\frac{h(t)}{t}\geq c_5
\end{equation}
and
\begin{equation}\label{25}
\lim\limits_{t\rightarrow\infty}\max\limits_{0\leq x\leq(c_5-\varepsilon)t}u(t,x)>0,
\end{equation}
where $c_5$ is given by (\ref{83}).
\end{theorem}

\begin{proof}
The inequality $d(1+c)<\lambda-b/m$ implies $c_3<c_4<c_1<c_2$. Choose $\varepsilon>0$ such that $c_4+\varepsilon<c_1-\varepsilon$. So
\begin{equation}\label{82}
g(t)<c_4+\varepsilon<c_1-\varepsilon<h(t),~~~~\forall~\rho\geq\rho_{\varepsilon},~\mu\geq\mu_{\varepsilon},~\tau>\tau_1.
\end{equation}

Note that $v\leq M_2$, where $M_2$ is given by Theorem \ref{th3}. By (\ref{10}), there exists $\tau_2\gg1$ such that $v(t,x)\leq\frac{2M_2}{\lambda-b/m}u(t,x)$ for all $t\geq\tau_2$ and $0\leq x\leq(c_1-\varepsilon)t$. Besides, by (\ref{9}) and (\ref{82}), we have $v(t,x)=0$ for $t\geq0$ and $x\geq(c_1-\varepsilon)t$. Denote $\kappa:=\frac{2M_2}{\lambda-b/m}$ and let $(\underline u,\underline h)$ be the unique solution of
$$
\begin{cases}
\underline u_t-\underline u_{xx}=(\lambda-\frac{b\kappa}{1+m\kappa})\underline u-\underline u^2,~~~~&t>\tau_2,~0<x<\underline h(t),\\
\underline u_x(t,0)=0,~\underline u(t,\underline h(t))=0,                                   &t\geq\tau_2,\\
\underline h^{\prime}(t)=-\mu \underline u_x(t,\underline h(t)),                            &t\geq\tau_2,\\
\underline u(\tau_2,x)=u(\tau_2,x),~\underline h(\tau_2)=h(\tau_2),&0\leq x\leq\underline h(\tau_2).
\end{cases}
$$
By the comparison principle, we have $\underline u(t,x)\leq u(t,x)$ and $\underline h(t)\leq h(t)$ for $t\geq\tau_2$ and $0\leq x\leq \underline h(t)$.  By Proposition \ref{p1}, we have
\begin{equation}\label{83}
\lim\limits_{\mu\rightarrow\infty}\lim\limits_{t\rightarrow\infty}\frac{\underline h(t)}{t}=\lim\limits_{\mu\rightarrow\infty}c(\mu,1,\lambda-\frac{b\kappa}{1+m\kappa})=2\sqrt{\lambda-\frac{b\kappa}{1+m\kappa}}:=c_5,
\end{equation}
which implies (\ref{26}) holds. Similar to the proof of (\ref{10}), we can show that
$$
\liminf\limits_{t\rightarrow\infty}\min\limits_{0\leq x\leq (c_5-\varepsilon)t}u(t,x)\geq\lambda-\frac{b\kappa}{1+m\kappa}>\lambda-\frac{b}{m}>0,
$$
which leads to (\ref{25}).
\end{proof}

\begin{theorem}\label{th10}
Suppose $\lambda<d$. Then the following hold,

{\upshape(i)}For any given $\varepsilon>0$, there exist $\mu_{\varepsilon}$, $T\geq1$ such that, when $\mu\geq\mu_\varepsilon$,
\begin{equation}\label{14}
\lim\limits_{t\rightarrow\infty}\sup\limits_{x\geq(c_3+\varepsilon)t}v(t,x)=0.
\end{equation}
Furthermore, if $d\leq2\sqrt{\lambda}+1$, then
\begin{equation}\label{15}
\lim\limits_{t\rightarrow\infty}\max\limits_{(c_2+\varepsilon)t\leq x\leq(c_3-\varepsilon)t}|v(t,x)-1|=0.
\end{equation}

{\upshape(ii)}There exists $\rho_0\gg 1$ such that when $\rho>\rho_0$,
\begin{equation}\label{16}
\lim\limits_{\rho\rightarrow\infty}\lim\limits_{t\rightarrow\infty}\frac{g(t)}{t}=c_3.
\end{equation}
\end{theorem}

\begin{proof}
(i) The assumption $\lambda<d$ implies $c_1<c_2<c_3<c_4$. So there exist $\tau_3\gg 1$ and $\mu_*\gg 1$ such that $h(t)\leq c_3t$ for all $t\geq\tau_3$ and $\mu\geq \mu_*$, which implies that $u(t,x)=0$ for all $t\geq\tau_3$, $x\geq c_3t$ and $\mu\geq\mu_*$. Choose $0<\varepsilon\ll1$ and define
$$
r(t)=\max\{(c_3+\varepsilon)t,g(t)\},~~~~for~t\geq\tau_3.
$$
Remember that $v(t,x)=0$ for $x\geq g(t)$ and $v_x(t,g(t))<0$, it is easy to verify that $v$ satisfies, in the weak sense,
$$
\begin{cases}
v_t-dv_{xx}\leq v-v^2,~&t\geq\tau_3,~c_3t\leq x\leq r(t),\\
v(t,x)\leq M_2,                &t\geq\tau_3,~c_3t\leq x\leq r(t),\\
v(t,r(t))=0,                 &t\geq\tau_3,
\end{cases}
$$
where $M_2$ is given by Theorem \ref{th3}. Define
$$
\xi(t,x)=M_2 e^{\frac{r(\tau_3)-c_3\tau_3}{\sqrt{d}}}e^{\frac{c_3t-x}{\sqrt{d}}},~~t\geq\tau_3,~c_3t\leq x<r(t).
$$
Obviously,
$$
\sup\limits_{x\geq (c_3+\varepsilon)t}\xi(t,x)\leq M_2e^{\frac{r(\tau_3)-c_3\tau_3}{\sqrt{d}}}e^{\frac{-\varepsilon t}{\sqrt{d}}}\rightarrow 0,~~t\rightarrow\infty,
$$
$$
\xi(t,c_3t)>M_2,~~\xi(t,r(t))>0,~~~~for~t\geq\tau_3,
$$
$$
\xi(\tau_3,x)>M_2,~~~~for~c_3\tau_3\leq x\leq r(\tau_3).
$$
Moreover, the direct calculations yield that
$$
\xi_t-d\xi_{xx}\geq\xi(1-\xi),~~~~for~t\geq\tau_3,~c_3t\leq x\leq r(t).
$$
By the comparison principle, we have $v(t,x)\leq\xi(t,x)$ for $t\geq\tau_3$ and $c_3 t\leq x\leq r(t)$. Thus, (\ref{14}) is obtained.

The proof of (\ref{15}). Combined with (\ref{11}), we just have to prove that
\begin{equation}\label{18}
\limsup_{t\rightarrow\infty}\max\limits_{x\geq c_2+\varepsilon}v(t,x)\leq 1.
\end{equation}
Choose $0<\varepsilon\ll1$  such that $c_2+\varepsilon/2\leq c_3-\varepsilon$. It follows from (\ref{8}) that $v$ satisfies
$$
\begin{cases}
v_t-dv_{xx}=v-v^2,~~~~&t\geq\tau_1,~(c_2+\varepsilon/2)t\leq x< g(t).\\
v(t,(c_2+\varepsilon)t)\leq M_2, ~v(t,g(t))=0,  &t\geq\tau_1.\\
v(\tau_1,x)\leq M_2,                              &(c_2+\varepsilon/2)\tau_1\leq x< g(\tau_1).
\end{cases}
$$
Define
$$
\phi(t,x)=1+M_2e^{g(\tau_1)}e^{(c_2+\varepsilon/2)t-x},~for~t\geq\tau_1,~(c_2+\varepsilon/2)t\leq x\leq g(t),
$$
Note that $d\leq2\sqrt{\lambda}+1$ and $\phi(t,x)\geq1$ for $t\geq\tau_1$, $(c_2+\varepsilon/2)t\leq x\leq g(t)$. Then $c_2+\frac{\varepsilon}{2}-d\geq2\sqrt{\lambda}-d\geq-1$. Thus, direct calculations yield that
$$
 \phi_t-d\phi_{xx}=(c_2+\frac{\varepsilon}{2}-d)(\phi-1)\geq\phi(1-\phi),~t\geq\tau_1,~(c_2+\varepsilon/2)t\leq x\leq g(t),
$$
$$
\phi(t,(c_2+\varepsilon/2)t)>1+M_2,~\phi(t,g(t))\geq1, ~t\geq\tau_1,
$$
$$
\phi(\tau_1,x)\geq 1+M_2,~(c_2+\varepsilon/2)\tau_1\leq x\leq g(\tau_1).
$$
By the comparison principle, we have $v(t,x)\leq\phi(t,x)$ for $t\geq\tau_1$ and $(c_2+\varepsilon/2)t\leq x\leq g(t)$. Thus,
\begin{equation}\label{19}
\max\limits_{x\geq(c_2+\varepsilon)t}v(t,x)=\max\limits_{(c_2+\varepsilon)t\leq x\leq g(t)}v(t,x)\leq\max\limits_{(c_2+\varepsilon)t\leq x\leq g(t)}\phi(t,x)=1+M_2e^{g(\tau_1)}e^{-\varepsilon t/2},
\end{equation}
which implies that (\ref{18}) holds.

{\upshape{(ii)}} For any given $0<\sigma\ll1$ and $\rho\geq\rho_{\varepsilon}$, where $\rho_{\varepsilon}$ is given by Theorem \ref{th5}. Let $(q_*(y),c_*)$ be the unique solution of (\ref{3}) with $(\beta,d,\theta)=(\rho,d,1+\sigma)$. Then $q_*^{\prime}(y)>0$,
$q_*(y)\rightarrow 1+\sigma$ as $y\rightarrow\infty$ and
\begin{equation}\label{85}
\lim\limits_{\rho\rightarrow\infty}c_*=2\sqrt{d(1+\sigma)}.
\end{equation}
Note that (\ref{19}), there exist $\rho_0>\rho_{\varepsilon}$, $\tau_0>\tau_1$, $y_0\gg1$ such that for all $\rho\geq\rho_0$, we have
$$
c_*>c_2+\varepsilon,~g(t)>(c_2+\varepsilon)t,~~~~\forall~t>\tau_0,
$$
$$
v(t,x)<(1+M_2e^{g(\tau_1)}e^{-\varepsilon t/2})q(y),~~\forall~t\geq\tau_0,~x\geq(c_2+\varepsilon)t,~y\geq y_0.
$$
Let $K=M_2e^{g(\tau_1)}$ and define
$$
\overline g(t)=c_*t+\vartheta K(e^{-\varepsilon\tau_0/2}-e^{-\varepsilon t/2})+y_0+g(\tau_0),~t\geq\tau_0,
$$
$$
\overline v(t,x)=(1+Ke^{-\varepsilon t/2})q_*(\overline g(t)-x),~t\geq\tau_0,~(c_2+\varepsilon)t\leq x\leq\overline g(t),
$$
where $\vartheta$ is to be determined. It is obvious that
$$
\overline g(\tau_0)>g(\tau_0),~\overline v(\tau_0,x)\geq v(\tau_0,x),~(c_2+\varepsilon)\tau_0\leq x\leq g(\tau_0),
$$
$$
\overline v(t,\overline g(t))=0=v(t,g(t)),~\overline v(t,(c_2+\varepsilon)t)>v(t,(c_2+\varepsilon)t),~t\geq\tau_0.
$$
Similar to the arguments of in \cite[Lemma 3.5]{Kaneko1}, we can calculate that, when $\vartheta$ is suitable large,
$$
\overline v_t-d\overline v_{xx}\geq\overline v(1-\overline v),~~t\geq\tau_0,~(c_2+\varepsilon)t\leq x<\overline g(t),
$$
$$
\overline g^{\prime}(t)\geq-\rho\overline v_x(t,\overline v(t)),~~~~t\geq\tau_0.
$$
Remember that $v_t-dv_{xx}=v-v^2$ for $t\geq\tau_0$ and $(c_2+\varepsilon)t\leq x<g(t)$, so we have $v(t,x)\leq\overline v(t,x)$ and $g(t)\leq\overline g(t)$ for all $t\geq\tau_0$ and $(c_2+\varepsilon)t\leq x<g(t)$ by the comparison principle. Therefore,
$$
\limsup_{t\rightarrow\infty}\frac{g(t)}{t}\leq c_*.
$$
By (\ref{85}) and the arbitrariness of $\sigma$, we have
$$
\limsup_{\rho\rightarrow\infty}\limsup_{t\rightarrow\infty}\frac{g(t)}{t}\leq 2\sqrt{d}=c_3.
$$
This combined with the second inequality of (\ref{17}) leads to (\ref{16}).
\end{proof}

\begin{theorem}\label{th11}
Assume that $0<m\lambda-b<b/c$ and denote $c_0:=\min\{c_1,c_3\}$. For any given $0<\varepsilon<c_0$, we have
$$
\lim\limits_{t\rightarrow\infty}\max\limits_{[0,(c_0-\varepsilon)t]}|u(t,x)-u^*|=0,~\lim\limits_{t\rightarrow\infty}\max\limits_{[0,(c_0-\varepsilon)t]}|v(t,x)-v^*|=0,
$$
where $u^*,~v^*$ is given by Theorem \ref{th7}.
\end{theorem}

\begin{proof}
The proof is inspired by \cite[Lemma 4.6]{25}. Define
$$
C_{[0,M]}:=\{(u,v)\in C^2(D_{\infty}^{h(t)})\times C^2 (D_{\infty}^{g(t)}):0\leq u\leq M_1,~0\leq v\leq M_2\},
$$
where $D_{\infty}^s$, $M_1$ and $M_2$ are given by Theorem \ref{th3}. Then $C_{[0,M]}$ is an invariant region of problem (\ref{Q}). By Theorem \ref{th5}, it is obvious that
$$
\lim\limits_{t\rightarrow\infty}\inf\limits_{0<x<(c_0-\frac{\varepsilon}{2})t}(u(t,x),v(t,x))>\mathbf{0},~\lim\limits_{t\rightarrow\infty}\inf\limits_{0<x<(c_0-\varepsilon)t}(u(t,x),v(t,x))>\mathbf{0}.
$$
Let $\{\varepsilon_n\}_{n=1}^{\infty}$ be a sequence with
$$
\frac{\varepsilon}{2}=\varepsilon_0<\varepsilon_1<\varepsilon_2<...<\varepsilon_n<...,~~\lim\limits_{n\rightarrow\infty}\varepsilon_n=\varepsilon.
$$
Define sequences $\overline u_n$, $\overline v_n$, $\underline u_n$, $\underline v_n$ as follows,
$$
\overline u_n=\lim\limits_{t\rightarrow\infty}\sup\limits_{0<x<(c_0-\varepsilon_n)t}u(t,x),~~\overline v_n=\lim\limits_{t\rightarrow\infty}\sup\limits_{0<x<(c_0-\varepsilon_n)t}v(t,x),
$$
$$
\underline u_n=\lim\limits_{t\rightarrow\infty}\inf\limits_{0<x<(c_0-\varepsilon_n)t}u(t,x),~~\underline v_n=\lim\limits_{t\rightarrow\infty}\inf\limits_{0<x<(c_0-\varepsilon_n)t}v(t,x).
$$
Then $\{\overline u_n\}$, $\{\overline v_n\}$ are monotonically non-increasing and$\{\underline u_n\}$, $\{\underline v_n\}$ are monotonically non-decreasing. Furthermore, there exist positive constant $\overline u$, $\overline v$, $\underline u$, $\underline v$ such that
$$
\lim\limits_{n\rightarrow\infty}(\overline u_n,\overline v_n,\underline u_n,\underline v_n)=(\overline u,\overline v,\underline u,\underline v).
$$
By \cite[Theorem 3.2]{25}, there exists a constant $\gamma>0$ such that
$$
\overline u_{n+1}\leq\overline u_n+\frac{\overline u_n}{\gamma}(\lambda-\overline u_n-\frac{b\underline v_n}{\overline u_n+m\underline v_n}),
$$
By letting $\varepsilon\rightarrow\infty$, we get $\lambda-\overline u-\frac{b\underline v}{\overline u+m\underline v}\geq0$. Similarity, we can verify that
$$
\lambda-\underline u-\frac{b\overline v}{\underline u+m\overline v}\leq0,~1-\overline v+\frac{c\overline u}{\overline u+m\overline v}\geq0,~1-\underline v+\frac{c\underline u}{\underline u+m\underline v}\leq0.
$$
Then $\overline u=\underline u=u^*$, $\overline v=\underline v=v^*$,
where $(u^*, v^*)$ is the positive equilibrium of (\ref{Q}) determined by (\ref{90}).
The proof is finished.
\end{proof}

\section{Discussion}
This paper mainly deals with the dynamics of a ratio-dependent type prey-predator model with free boundaries $x=h(t)$ and $x=g(t)$ which describe the spreading fronts of prey and predator, respectively. Here, the two free boundaries $x=h(t)$ and $x=g(t)$ may intersect each other as time goes on. So their dynamics are rather complicated.

In Section 3, we get the long time behaviors of two species as follows:

(i) if the prey (predator) cannot spread to the whole space, then it will be vanishing eventually; (Theorem \ref{th6})

(ii) if two species can spread successfully, then they will stabilize at a positive equilibrium state; (Theorem \ref{th7})

(iii) if the prey (predator) spreads successfully and the predator (prey) cannot spread into $[0,\infty)$, then the former will stabilize at a positive point and the latter will vanish eventually. (Theorem \ref{th8})

Section 4 gives some conditions for spreading and vanishing about prey and predator which are summarized as follows.

(i) If one of the initial habitat and the moving parameter of the prey (predator) is``suitably large", then it is always able to spread successfully. While when both the initial habitat and the moving parameter of the prey (predator) are``suitably small", the prey (predator) will vanish eventually; (Theorem \ref{th1})

(ii) If the predator spreads slowly and the initial habitat of prey is much larger than that of predator, then the prey's territory always covers that of the predator and the prey will spread successfully no matter whether the predator successfully spreads or not (Theorem \ref{th2}).

(iii) Assume that the prey spreads slowly and the predator does quickly, if $\lambda^2-(b-m\lambda)<0$ and the predator spreads successfully, then the prey will vanish eventually.(Theorem \ref{th4})

The asymptotic spreading speed and asymptotic speed are studied in Section 5, whose conclusions show the complicated and realistic spreading phenomena of prey and predator. To better understand these dynamics, we first provide the spreading speed of the prey and predator when problem (\ref{Q}) is uncoupled($b=c=0$):
\begin{equation}\label{86}
\begin{cases}
u_t-u_{xx}=\lambda u-u^2,~~~~&t>0,~0<x<h(t),\\
u_x(t,0)=u(t,h(t))=0,        &t\geq0,\\
h^{\prime}(t)=-\mu u_x(t,h(t)),&t\geq0,\\
u(0,x)=u_0(x),~h(0)=h_0,        &0\leq x\leq h_0.
\end{cases}
\end{equation}
\begin{equation}\label{87}
\begin{cases}
v_t-dv_{xx}=v-v^2,~~~~&t>0,~0<x<g(t),\\
v_x(t,0)=v(t,g(t))=0,        &t\geq0,\\
g^{\prime}(t)=-\rho v_x(t,g(t)),&t\geq0,\\
v(0,x)=v_0(x),~g(0)=g_0,        &0\leq x\leq g_0.
\end{cases}
\end{equation}
By Proposition \ref{p1}, we have
$$
\lim\limits_{\mu\rightarrow\infty}\lim\limits_{t\rightarrow\infty}\frac{h(t)}{t}=2\sqrt{\lambda},
~\lim\limits_{\mu\rightarrow\infty}\lim\limits_{t\rightarrow\infty}\frac{g(t)}{t}=2\sqrt{d},
$$
which shows the asymptotic speed of the prey is $2\sqrt{\lambda}$ and that of the predator is $2\sqrt{d}$ when problem (\ref{Q}) is uncoupled.

(i) Assume that $d(1+c)<\lambda-b/m$. Then the asymptotic speed of prey is between $2\sqrt{\lambda-\frac{b\kappa}{1+m\kappa}}$ and $2\sqrt{\lambda}$ and that of predator is between $2\sqrt{d}$ and $2\sqrt{d(1+c)}$ as $\mu$, $\rho\rightarrow\infty$. Those manifest that the predator could decrease the prey's asymptotic speed, but, conversely, the prey could accelerate that of predator. Besides, when we are to move to the right at a fixed  speed less than $2\sqrt{d}$, we will observe that the two species will stabilize at the unique positive equilibrium; when we do that with a fixed speed between $2\sqrt{d(1+c)}$ and $2\sqrt{\lambda-\frac{b\kappa}{1+m\kappa}}$, we will only see the prey; when we do that with a fixed speed over than $2\sqrt{\lambda}$, we will see neither.(Theorem \ref{th5}, Theorem \ref{9}, Theorem \ref{th11})

(ii)Assume that $m\lambda>b$ and $\lambda<d$. Then the asymptotic speed of prey is between $2\sqrt{\lambda-b/m}$ and $2\sqrt{\lambda}$ and that of predator is $2\sqrt{d}$ when $\mu$, $\rho\rightarrow\infty$, which illustrates that the predator could decrease the asymptotic speed of the prey, while the prey has no effect on the predator. In addition, if we are to move to the right at a fixed speed less than $2\sqrt{\lambda-b/m}$, then we will see that the two species will stabilize at the unique positive equilibrium; if we do that with a fixed speed between $2\sqrt{\lambda}$ and $2\sqrt{d}$, then we will only observe the predator; if we do that with a fixed speed more than $2\sqrt{d}$, then we will see neither.(Theorem \ref{th5}, Theorem \ref{th10}, Theorem \ref{th11})

By comparing the ratio-dependent type(\ref{Q}) and the logistic type with the same free boundary mechanism, we find some differences. The main difference is that when the the speed of predator is slower than that of prey (i.e., $d(1+c)<\lambda-\frac{b}{m}$), the predator could decrease the prey’s asymptotic speed in the setting model studied in this paper; while in \cite{wang12} the former is not harmful to the latter. The phenomenon in this paper seems closer to reality.

\section*{Acknowledgments}
The author's work was supported by NSFC (12071316). The author is also grateful for the anonymous reviewers for their helpful comments and suggestions.

\section*{Declarations of interest: none.}
\section*{References}
\bibliographystyle{plain}\setlength{\bibsep}{0ex}
\scriptsize
\bibliography{mybibfile}

\begin{thebibliography}{10}

\bibitem{Du6}
G.~Bunting, Y.H. Du, and K.~Krakowski.
\newblock Spreading speed revisited: Analysis of a free boundary model.
\newblock {\em Netw. Heterog. Media}, 7(4):583--603, 2012.

\bibitem{15}
X.F. Chen and A.~Friedman.
\newblock A free boundary problem arising in a model of wound healing.
\newblock {\em SIAM J. Math. Anal.}, 32(4):778--800, 2000.

\bibitem{Du2}
Y.H. Du and Z.M. Guo.
\newblock Spreading--vanishing dichotomy in a diffusive logistic model with a
  free boundary, ii.
\newblock {\em J. Differ. Equations}, 250(12):4336--4366, 2011.

\bibitem{2012Du}
Y.H. Du and Z.M. Guo.
\newblock The {S}tefan problem for the {F}isher–-{KPP} equation.
\newblock {\em J. Differ. Equations}, 253(3):996--1035, 2012.

\bibitem{Du3}
Y.H. Du, Z.M. Guo, and R.~Peng.
\newblock A diffusive logistic model with a free boundary in time--periodic
  environment.
\newblock {\em J. Funct. Anal.}, 265(9):2089--2142, 2013.

\bibitem{Du2001LOGISTIC}
Y.H. Du and M.~Li.
\newblock Logistic type equations on ${R}^n$ by a squeezing method involving
  boundary blow--up solutions.
\newblock {\em J. Lond. Math. Soc.}, 64(1), 2001.

\bibitem{Du1}
Y.H. Du and Z.G. Lin.
\newblock Spreading--vanishing dichotomy in the diffusive logistic model with a
  free boundary.
\newblock {\em SIAM J. on Math. Anal}, 42(1):377--405, 2010.

\bibitem{Du5}
Y.H. Du and Z.G. Lin.
\newblock The diffusive competition model with a free boundary: invasion of a
  superior or infetior competitor.
\newblock {\em Discrete Contin. Dyn. Syst.}, 19(10):3105--3132, 2014.

\bibitem{Du4}
Y.H. Du and B.D. Lou.
\newblock Spreading and vanishing in nonlinear diffusion problems with free
  boundaries.
\newblock {\em Mathematics}, 17(10):2673--2724, 2015.

\bibitem{Kaneko1}
Y.~Kaneko and H.~Matsuzawa.
\newblock Spreading speed and sharp asymptotic profiles of solutions in free
  boundary problems for nonlinear advection--diffusion equations.
\newblock {\em J. Math. Anal. Appl.}, 428(1):43--76, 2015.

\bibitem{25}
G.~Lin.
\newblock Spreading speeds of a {L}otka--{V}olterra predator-prey system: The
  role of the predator.
\newblock {\em Nonlinear Anal.}, 74(7):2448--2461, 2011.

\bibitem{liu2020free}
L.Y. Liu.
\newblock A free boundary problem for a ratio-dependent predator-prey system.
\newblock {\em arXiv preprint arXiv:2006.13770}, 2020.

\bibitem{L1971The}
L.I. Rubenstein.
\newblock The {S}tefan problem.
\newblock {\em Amer. Math. Soc., Providence, RI}, 1971.

\bibitem{wang11}
M.X. Wang.
\newblock On some free boundary problems of the prey--predator model.
\newblock {\em J. Differ. Equ.}, 256(10):3365--3394, 2014.

\bibitem{wang10}
M.X. Wang.
\newblock Spreading and vanishing in the diffusive prey--predator model with a
  free boundary.
\newblock {\em Commun. Nonlinear Sci. Numer. Simul.}, 23(1-3):311--327, 2015.

\bibitem{wang9}
M.X. Wang and Y.~Zhang.
\newblock Two kinds of free boundary problems for the diffusive prey--predator
  model.
\newblock {\em Nonlinear Anal.-Real}, 24:73--82, 2015.

\bibitem{wang12}
M.X. Wang and Y.~Zhang.
\newblock Dynamics for a diffusive prey--predator model with different free
  boundaries.
\newblock {\em J. Differ. Equ.}, 264(5):3527--3558, 2018.

\bibitem{wang7}
M.X. Wang and J.F. Zhao.
\newblock Free boundary problems for a lotka--volterra competition system.
\newblock {\em J. Dyn. Differ. Equ.}, 26(3):655--672, 2014.

\bibitem{wang6}
M.X. Wang and J.F. Zhao.
\newblock A free boundary problem for the predator--prey model with double free
  boundaries.
\newblock {\em J. Dyn. Differ. Equ.}, 29(3):957--979, 2017.

\bibitem{wang8}
M.X. Wang and Y.G. Zhao.
\newblock A semilinear parabolic system with a free boundary.
\newblock {\em Z. Angew. Math. Phys.}, 66(6):3309--3332, 2015.

\bibitem{Wangzhao2014a}
J.F. Zhao and M.X. Wang.
\newblock A free boundary problem of a predator--prey model with higher
  dimension and heterogeneous environment.
\newblock {\em Nonlinear Anal.- Real}, 16:250--263, 2014.

\bibitem{wang5}
Y.G. Zhao and M.X. Wang.
\newblock A reaction--diffusion--advection equation with mixed and free
  boundary conditions.
\newblock {\em J. Dyn. Differ. Equ.}, 2015.

\end{thebibliography}
\end{document}